\documentclass[a4paper,11pt,oneside]{article}
\usepackage{amsmath,amsthm,amsfonts,amssymb,ifpdf}
\usepackage{amssymb}
\usepackage{amsmath}
\usepackage{amsthm}
\usepackage{graphicx}
\usepackage[all]{xy}
\usepackage{enumerate}
\usepackage{tikz-cd}
\usepackage{subcaption}
\usepackage{authblk}% for dual author

\ifpdf        % color hyper link
\usepackage[
pdftex,
colorlinks,%
linkcolor=blue,citecolor=red,urlcolor=blue,
hyperindex,%
plainpages=false,%
bookmarksopen,%
bookmarksnumbered%
]{hyperref} 
\usepackage{thumbpdf}
\else
\usepackage{hyperref}
\fi

\usepackage{theoremref} % thoerem ref

\usepackage{tikz}
\usetikzlibrary{positioning, arrows.meta}

\newtheorem*{claim*}{Claim}

\newtheorem{theorem}{Theorem}[section]
\newtheorem{lemma}{Lemma}[section]

\newtheorem{proposition}{Proposition}[section]
\newtheorem{remark}{Remark}[section]
\newtheorem{definition}{Definition}[section]

\numberwithin{equation}{section}
\title{Decay estimates and blow up of solutions to a class of heat equations}
\author{Joydev Halder\thanks{halderjoydev@gmail.com}\ }
\author{Suman Kumar Tumuluri \thanks{suman.hcu@uohyd.ac.in}}

\affil{School of Mathematics and Statistics, University of Hyderabad, Hyderabad, India.}

\date{\today}
\begin{document}
\maketitle

\begin{abstract}
In this article, we study a semi-linear heat equation  with the nonlinearity which is the product of polynomial and logarithmic functions. Using the invariance of the potential well(s), we have established the global existence and exponential decay estimates of solutions in $L^2$ - norm  without having any restriction on the exponent in the source term under suitable conditions on the initial data. Moreover, finite time blow up of solutions at subcritical, critical and supercritical initial energy levels is also discussed.
\end{abstract}

\section{Introduction}\label{sec:introduction}
Partial differential equations with polynomial and logarithmic nonlinearities are studied widely due to many applications in physics and other applied sciences such as  transport and diffusion phenomena, 
 nuclear physics, and theory of superfluidity etc.,  (see \cite{chen2012_3, chen2015_258,gazzola2005_18,chao2016_437,lian2020_40,yacheng2006_64,zloshchastiev2010_16}). In this article, we study the following initial value problem of a semilinear heat equation with the nonlinearity which is the product of logarithmic and  polynomial functions:
\begin{equation}\label{main}
	\left\{
	\begin{aligned}
		&v_t-\Delta v=v|v|^{p-1}\log|v|, &&x \in {U}, t>0,
		\\
		&v(x,t)=0, &&x \in \partial {U}, t>0,
		\\
		&v(x,0)=v_0(x), && x \in {U},
	\end{aligned}
	\right.
\end{equation}
where ${U} \subset \mathbb{R}^n (n \geq 1)$, is a smooth domain and $v_0 \in H_0^1({U})$.\\
In the literature, many authors investigated the global existence and blow up phenomena of the solutions to the equations of the following form
\begin{equation}\label{parabolic_equation}
	\left\{
	\begin{aligned}
		&v_t-\Delta v=f(v), &&x \in {U}, t>0,
		\\
		&v(x,t)=0, &&x \in \partial {U}, t>0,
		\\
		&v(x,0)=v_0(x), && x \in {U},
	\end{aligned}
	\right.
\end{equation}
for different choices of $f$ using various methods (see for instance  \cite{ambrosetti1973_14,hoshino1991_34,stetter1973_23,tsutsumi1972_17,weissler1979_32} and the references therein). The ``potential well method", developed by Sattinger, is a powerful technique for studying the global existence and finite time blow up of solutions to \eqref{parabolic_equation} (see \cite{ payne1975_22,sattinger1968_30}).
% A powerful technique to study the global existence and finite time blow up of solution to \eqref{parabolic_equation} is the ``potential well method'', due to Sattinger (see \cite{ payne1975_22,sattinger1968_30}). 
Payne and Sattinger's work \cite{payne1975_22} is the most influential one which is followed by many mathematicians. In \cite{payne1975_22}, the authors  introduced the potential well $W$, the outer potential well $V$, and the depth $\delta$ of the potential well in terms of a Nihari functional $I(v(\cdot, t))$ and the energy functional $J(v(\cdot, t))$ associated to the solution $v(x,t)$ of the PDE that was considered. Using the invariance of the potential well,  they established the finite time blow up of the solutions to \eqref{parabolic_equation}. After that, the method of potential well was improved by numerous authors to analyze different types of PDEs such as nonlinear heat equations, pseudo parabolic equations and nonlinear wave equations (see \cite{chen2012_3,chen2015_422,chen2015_258,han2019_474,han2019_164,lian2020_40,liu2020_28,yacheng2003_192,yacheng2006_64}). Particularly,   Yacheng introduced the notion of a family of the potential wells in \cite{yacheng2006_64}, in which the potential well $W$ is a member. Furthermore, the authors of \cite{gazzola2005_18} studied  problem \eqref{parabolic_equation}, when $f(v)=v|v|^{p-1}$ and used the properties of the family of potential wells to prove the global existence and finite time blow-up of  solutions. Moreover,  the authors of \cite{chen2012_3} investigated the existence of global  solutions with an exponential decay to a class of degenerate parabolic equations with the source term $f(v)=v|v|^{p-1}$. Furthermore, they showed that the solutions  exhibit finite time blow up under  suitable conditions.  In \cite{chen2015_422}, the authors considered  \eqref{parabolic_equation} with $f(v)=v\log|v|$. Using the an appropriate family of potential wells and the logarithmic Sobolev inequality, they established the  global existence  and blow up at $t=+\infty$.  

On the other hand, many authors studied the nonlinear wave equation using the potential wells (see \cite{lian2020_40,liu2020_28,payne1975_22,yacheng2003_192,yacheng2006_64}). In particular, authors of \cite{yacheng2006_64} established the global existence of the solution to nonlinear wave equation under some assumption on initial energy and source term $f(v)$. The authors of \cite{lian2020_40} recently studied the nonlinear wave equation with source term $f(v)=|v|^p \log |v|$. %They proved the global existence and the finite time blow up of the solutions when the initial energy is subcritical, critical and supercritical.
They showed the existence of global solutions  at subcritical and critical initial enegry levels. They also established that solutions exhibit finite time blow up property for all the enegy levels. Moreover, many authors used this method to show existence of global solutions and finite time blow up of solutions in the context of semilinear pseudo-parabolic equations with nonlinear source terms  (see \cite{chen2015_258,nhan2017_73,xu2013_264}).
 %Moreover, in the context of semilinear pseudo-parabolic equation with nonlinear source term, many authors used this method to prove the global existence and blow up in finite time (see \cite{chen2015_258,nhan2017_73,xu2013_264}).

In \cite{han2019_164,le2017_151}, the authors considered a class of  $p$-Laplacian heat equations on a bounded domain ${U} \subset \mathbb{R}^n\ (n \geq 2)$ with source term $f(v)=v|v|^{p-2}\log|v|,\  p \in (2,\infty)$. The authors of \cite{le2017_151} proved the  global existence and   blow up  in finite time  of solution to \eqref{parabolic_equation} under the condition $J(v_0)< \delta$. In \cite{han2019_164}, the authors established that solutions  blow up in  finite time  for large values of $J(v_0)$. Recently in \cite{peng2021_100}, the authors considered \eqref{main} to investigate the global existence and blow up phenomena of its solutions. They indeed obtained the global existence and shown blow up with a severe restriction on the exponent $p$ in \eqref{main}. For instance, their results are not applicable when the source term in \eqref{main} is $v|v|^4 \log |v|$ and $n \geq 3$. In the present article, we have proved the global existence of solutions and shown that the solution exhibits finite time blow up without such restriction on $p$.   Moreover, we have obtained  the decay estimates of the $L^2$ - norm of the global solutions  without having any restriction on $p$ and the dimension $n$. Furthermore, we have estimated the $H_0^1$-norm of the solution with some condition on $p$ and $n$. The methods to prove the decay estimates of $H^1_0$ norm of the solution that are discussed in detail in our paper can be extended to pseudo parabolic equations also (see \cite{halder_pseudo}). 
%However, when $f(v)=v|v|^{p-1}\log |v|$, as far as we know, there is no results on this aspect of the global existence and blow up for the problem \eqref{main}.

This paper is organized as follows. Preliminaries, as well as the modified family of potential wells are discussed in Section \ref{sec:Preliminaries}. Global existence and finite time blow up of weak solutions to the problem \eqref{main} under the condition $J(v_0) <\delta$ are discussed in Section \ref{sec:Subcritical case}. 
Similar results are found in Section \ref{sec:The critical case} when $J(v_0) = \delta$.  Finally, in Section \ref{sec:finite_time_blow}, we presented the finite-time blow-up under the condition $J (v_0) > \delta$.

%This paper is organized as follows. In Section \ref{sec:Preliminaries}, we introduce preliminaries, and the  modified  family of  potential wells. In Section \ref{sec:Subcritical case}, global existence and finite time blow up of weak solutions to the problem \eqref{main} under the condition  $J(v_0) < \delta$ are presented. Similar results under the condition $J(v_0) = \delta$ are established in Section \ref{sec:The critical case}. Finally, we study the finite time blow up in Section \ref{sec:finite_time_blow}.
%%%%%%%%%%%%%%%%%%%%%%%%%%
\section{Preliminaries}\label{sec:Preliminaries}
%Henceforth, the norm $||\cdot||_p$ is denoted by $||g||_p = (\int_{U} |g|^pdx )^{\frac{1}{p}}, \forall g \in L^p({U})$, $1 \leq p < \infty$ and instead of $|| \cdot||_2$, we write $|| \cdot ||$. the inner product $(h,g)_2=\int_{U} h g dx$.
For $ 1 \leq p < \infty$, $g \in L^p({U})$, we denote the $L^p$- norm of $g$ by $||g||_p = (\int_{U} |g|^pdx )^{\frac{1}{p}},$ and the $L^2$-inner product by $(h,g)_2=\int_{U} h g dx$, $\forall g,h \in L^({U})$. If $p=2$, we simply write $||g ||$ instead of $|| g ||_2$.
\begin{definition}[Weak solution]
	A function $v(x,t)$ is called a weak solution to problem \eqref{main} on ${U} \times [0,T) $, if $v \in L^\infty (0,T; H^1_0({U}))$ with $v_t \in  L^2 (0,T; L^2 ({U}))$ such that  $v(x,0)=v_0(x)$  and satisfies  \eqref{main} in the distribution sense, i.e.,\\ \medskip \centerline{$(\nabla v, \nabla w)_2+(v_t, w)_2=(v|v|^{p-1}\log|v|, w)_2$,}\\
	for every  $t\in(0,T)$, $w\in H^1_0({U})$.
\end{definition}
\noindent
For completeness, we recall here the   definitions of the maximal existence time, and the notion of the finite time blow up which are quite standard.
\begin{definition}[Maximal existence time]
	The maximal time of existence $T$ of a weak solution  $v(x,t)$ to \eqref{main} is defined as follows:\\
	$(i)$ If $v(x,t)$ exists for all $0\leq t < \infty$, then $T=\infty$.
	\\
	$(ii)$ If there exists  $\tilde{t} \in (0,\infty)$ such that $v(x,t)$ exists for $0\leq t <\tilde{t}$, but does not exist at $t=\tilde{t}$, then $T=\tilde{t}$.
\end{definition}
\begin{definition}[Finite time blow-up]
	A weak solution $v$ of \eqref{main} is said to blow up in finite time if the maximal existence time $T$ is finite and $$\lim\limits_{t \to T^-} ||v(\cdot,t)||=\infty.$$
\end{definition}
%%%%%%%%%%%%%%%%%%%%%%%%%%%%%%%%%%%%%%%%%%%%%%%%%
%%%%%%%%%%%%%%%%%%%%%%%%%%%%%%%%%%%%%%%%%%%%%%%%%
\subsection{Potential wells}
In this subsection, we define the family of the potential wells and the energy functionals corresponding  to the nonlinear term $f(v)=v|v|^{p-1}\log|v|$.\\
First, we define  the Nehari functional $I$ and the potential energy functional $J$  as follows:
\begin{equation}\label{nehari}
	I(v)=||\nabla v||^2- \int\limits_{{U}}|v|^{p+1}\log |v| dx,
\end{equation}
and
\begin{equation}\label{potential}
	J(v)=\frac{1}{2}||\nabla v||^2- \frac{1}{p+1}\int\limits_{{U}}|v|^{p+1}\log |v| dx+ \frac{1}{(p+1)^2}||v||_{p+1}^{p+1}.
\end{equation}
Observe that 
\begin{equation}\label{combo}
	J(v)=\frac{p-1}{2(p+1)}||\nabla v||^2+ \frac{1}{p+1}I(v)+  \frac{1}{(p+1)^2}||v||_{p+1}^{p+1}.
\end{equation}
Moreover, the Nehari manifold is defined as 
\begin{equation*}
	\mathcal{N}(v)=\{v \in H^1_0({U})\mid  ||\nabla v||^2 \neq 0, I(v)=0\},
\end{equation*}
and the depth of the well is
\begin{equation*}
	\delta=\inf\limits_{v\in \mathcal{N}} J(v).
\end{equation*}
We now introduce the outer potential well
\begin{equation*}
	V=\{v\in H_0^1({U}) \mid J(v)<\delta, I(v)<0 \},
\end{equation*}
and the potential well
\begin{equation*}
	W=\{v\in H_0^1({U}) \mid  J(v)<\delta, I(v)>0 \} \cup \{0\}.
\end{equation*}
Observe that, if $v$ is a weak solution to \eqref{main}, then on multiplying \eqref{main} with $v_t$ and integrating over ${U} \times [0,t)$, we find that 
\begin{equation}\label{energy_inequality}
	J(v(\cdot,t))+\int\limits_{0}^{t} ||v_t(\cdot,\tau)||^2 d \tau=J(v_0),\  t\in [0,T).
\end{equation}
Set the associated energy functional to \eqref{main} as follows:
\begin{equation*}
	E(v(\cdot,t))= \int\limits_{0}^t ||v_t(\cdot, \tau)||^2 d \tau+J(v(\cdot,t)).
\end{equation*}
In view of \eqref{energy_inequality}, we immediately get $E(v(\cdot,t)) = E(v_0)$. Henceforth  we  refer to   \eqref{energy_inequality} as the conservation of energy.\\
Next, we define a family of potential energy functionals to extend the concept of a single potential well to a family of potential wells by
%Next, we extend the notion of a single potential well to the family of potential wells by defining a family of potential energy functionals
\begin{equation}\label{d_potential}
	J_{\rho}(v)=\frac{\rho}{2}||\nabla v||^2- \frac{1}{p+1}\int\limits_{{U}}|v|^{p+1}\log |v| dx+ \frac{1}{(p+1)^2}||v||_{p+1}^{p+1},
\end{equation}
where $\rho >0$.
Moreover, we define a family of Nehari functional
\begin{equation}\label{d_nehari}
	I_{\rho}(v)=\rho ||\nabla v||^2- \int\limits_{{U}}|v|^{p+1}\log |v| dx,
\end{equation}
the corresponding Nehari manifolds
\begin{equation*}
	\mathcal{N}_{\rho}(v)=\{v \in H^1_0({U})\mid ||\nabla v||^2 \neq 0, I_{\rho}(v)=0\},
\end{equation*}
and depth of family potential wells
\begin{equation}\label{d_delta}
	\delta(\rho)=\inf\limits_{v\in \mathcal{N}_\rho} J(v).
\end{equation}
We also introduce the outer of the family of potential wells 
\begin{equation*}
	V_\rho=\{v\in H_0^1({U}) \mid  J(v)<\delta(\rho), I_\rho(v)<0 \},
\end{equation*}
and  the family of potential wells 
\begin{equation*}
	W_\rho=\{v\in H_0^1({U}) \mid  J(v)<\delta(\rho), I_\rho (v)>0 \} \cup \{0\}.
\end{equation*}

Since the functionals $I$ and $J$ defined in  \eqref{nehari}--\eqref{potential} are the same as those in \cite{lian2020_40} (in the context of semilinear wave equation), we recall the results proved in that paper which can be used in our analysis.
%%%%%%%%%%%%%%%%%%%%%%%
\begin{lemma}(Cf. \cite{lian2020_40}, Lemma 2.1)\label{lemma1}
	Set $g(\zeta)=J(\zeta v)$, then for any nonzero $v \in H_0^1 ({U})$ we  have  the following:\\
	$(i)$  $\lim\limits_{\zeta \to \infty} J(\zeta v)=- \infty$, $\lim\limits_{\zeta \to 0} J(\zeta v)=0$.\\
	$(ii)$ There exists a unique $\zeta^* = \zeta ^*(v)$ in the interval $(0, \infty)$ such that
	$$\frac{d}{d \zeta} J(\zeta v) \mid_{\zeta= \zeta^* }=0.$$
	$(iii)$ The function $\zeta \mapsto J(\zeta v)$ is  decreasing on $\zeta^* \leq \zeta <\infty$, increasing on $0 \leq \zeta \leq  \zeta^*$, and attains its maximum at $\zeta=\zeta^*$.\\
	$(iv)$ The function $I$ satisfies $I(\zeta v)=\zeta \frac{d}{d \zeta}J(\zeta v)< 0$ when $\zeta^* < \zeta < \infty$, $I(\zeta v) >0$ when $0<\zeta <\zeta^*$, and $I(\zeta^*v)=0$.
\end{lemma}
%%%%%%%%%%%%%%%%%%%%%%%
\begin{lemma}(Cf. \cite{lian2020_40}, Lemma 2.3)\label{lemma3}
	The function $\rho \mapsto \delta(\rho)$ defined in \eqref{d_delta} has the following properties:\\
	$(i)$ There exists a unique $\rho_0>\frac{p+1}{2}$ such that  $\delta(\rho)>0$ for $0<\rho<\rho_0$, and $\delta(\rho_0)=0$.
	\\
	$(ii)$ The function $\rho \mapsto \delta(\rho)$ is strictly decreasing on $1\leq \rho \leq \rho_0$, increasing on $0 <\rho \leq 1$, and  at $\rho=1$ this function attains a local maximum, and $\delta(1)=\delta$.
\end{lemma}
%%%%%%%%%%%%%%%%%%%%%%%%%%%%%%%%%%%%%%%%%%%%%%%
%\subsection{Invariant sets}
\begin{lemma}\label{lemma_inv}(Cf. \cite{lian2020_40}, Lemma 2.4)
	If $0<J(v)<\delta$ for some $v \in H^1_0({U})$, then the sign of $I_\rho(v)$ does not change in $\rho_1<\rho < \rho_2$, where $\rho_1 <1< \rho_2$ are two roots of the equation $J(v)=\delta(\rho)$.
\end{lemma}
\begin{remark}
	To prove Lemma \ref{lemma1} -- \ref{lemma_inv}, we do not need the Gagliardo–Nirenberg–Sobolev inequality.
\end{remark}
\begin{lemma}\label{lemma_ineq}(Cf. \cite{komornik1997_39}, Theorem 8.1)
	Assume $f: \mathbb{R}^+ \to \mathbb{R}^+$ be a non-increasing function, and there exists a positive constant
	$ K $ such that 
	\begin{equation}\label{f_es}
		\int\limits_{t}^\infty f(\tau)d\tau <Kf(t), \ t \geq 0,
	\end{equation}
	then $f(t) \leq f(0) e^{1-t/K}$, $t \geq 0$.
\end{lemma}
\begin{lemma}\label{lemma_log}(Cf. \cite{nhan2017_73}, Lemma 2.1)
	For any positive number $\theta$, we have the inequality
	\begin{equation}\label{lem_log}
		\log x \leq \frac{x^\theta}{e \theta},
	\end{equation}
where $x \in [1,\infty)$.
\end{lemma}
Now we are ready to prove the following invariant property of the sets $W_\rho$ and $V_\rho$.
\begin{theorem}[Invariant sets]\label{invariant}
	Assume  that  $0<\eta<\delta$ and $v_0 \in H^1_0({U})$. Let $\rho_1<\rho_2$ are the two roots of the equation $\delta(\rho)=\eta$, then the following hold true. \\
	$(i)$ If $I(v_0)>0$ and   $0<J(v_0)\leq \eta$, then all  weak solutions to problem \eqref{main}  belong to $W_\rho$ for $\rho_1<\rho<\rho_2$.
	\\
	$(ii)$ If $I(v_0)<0$ and   $0<J(v_0)\leq \eta$, then all  weak solutions to problem \eqref{main}  belong to $V_\rho$ for $\rho_1<\rho<\rho_2$.
\end{theorem}
\begin{proof}
	$(i)$ Let $v$ be a weak solution to \eqref{main}  with $I(v_0)>0$, and $0 < J(v_0) \leq \eta <\delta$. Suppose $v(\cdot,t)$ exists in $[0,T)$. In view of Lemma \ref{lemma3} $(iii)$,  we deduce that $\rho_1<1 <\rho_2$. Moreover, from Lemma \ref{lemma_inv}, one can easily get that $I_\rho(v_0) >0$, since $I(v_0)>0$. Therefore for $\rho_1 < \rho < \rho_2$, we obtain $v_0 \in W_\rho$. Next we prove that for any $\rho \in (\rho_1,\rho_2)$, we have $v(\cdot,t) \in W_\rho$, $0<t<T$. On the contrary, we assume that for $\rho_1 < \rho < \rho_2$ there exists $t_0 \in (0,T)$ such that $v(\cdot,t_0) \in \partial W_\rho$. Hence, either  $J(v(\cdot,t_0))=\delta(\rho)$ or $I_\rho (v(\cdot,t_0))=0$ holds. From the conservation of energy \eqref{energy_inequality}, we have
	\begin{equation}\label{energy_equ}
		J(v(\cdot, t))+\int\limits_{0}^{t} ||v_t||^2 dt =J(v_0) \leq \eta < \delta(\rho),\ 0<t<T, \  \rho_1 < \rho < \rho_2,
	\end{equation}
	which is a contradiction to $J(v(\cdot,t_0))=\delta(\rho)$. On the other hand, if $I_\rho (v(\cdot,t_0))=0$, then by the definition  of $\delta(\rho)$, we obtain $J(v(t_0))\geq \delta(\rho)$, which is a contradiction to conservation of energy \eqref{energy_equ}.\\
	$(ii)$  One can easily prove $(ii)$ using the same arguments used in the proof of $(i)$, and the details are omitted.
	%Using the similar arguments used in the proof of  $(i)$, one can easily prove $(ii)$ and the details are omitted.
\end{proof}
%%%%%%%%%%%%%%%%%%%%%%%%%%%%%%%%%%%%%%%%%
%%%%%%%%%%%%%%%%%%%%%%%%%%%%%%%%%%%%%%%%
\section{Subcritical case $J(v_0)<\delta$}\label{sec:Subcritical case}
In this section, we establish the global existence, and finite time blow up of the solutions to \eqref{main} under subcritical initial energy level $J(v_0) <\delta$ depending on the sign of $I(v_0)$. We first consider the case $I(v_0)>0$ and establish the global existence of solution. We also prove decay estimates whenever global solutions exist.
%\section{Global existence}
%\subsection{Global existence at $J(v_0)<\delta$}
\begin{theorem}\label{them1}
	Assume that  $I(v_0)>0$, $J(v_0)<\delta$  then problem \eqref{main} admits a global weak solution $v \in L^\infty (0,\infty; H^1_0({U}))$. Moreover, we have $v(\cdot, t) \in W$ for $0\leq t<\infty$, and there exist constants $\gamma >0$ and $C>0$  such that
	\begin{equation}\label{decay}
		||v(\cdot,t)||\leq C e^{-\gamma t}, \quad 0 \leq t < \infty.
	\end{equation}
\end{theorem}
\begin{proof}The proof is divided into two steps. \\
	%We devide the proof into two steps.\\
	\textbf{Step 1:} In this step, we prove global existence of a solution to \eqref{main}.\\
	Let $\{w_j(x)\}_{j=1}^{\infty}$ be an orthogonal basis of $H^1_0({U})$. Using the Galerkin method employed in \cite{chen2012_3, yacheng2006_64}, we can construct a sequence $(v_m)$ of  approximate solutions to \eqref{main} given by
	\begin{equation}\label{approx1}
		v_m(x,t)=\sum\limits_{j=1}^m  w_j(x) g_{jm}(t), \ m=1,2,\dots,
	\end{equation}
	which satisfy
	\begin{equation}\label{approximate}
		(\frac{\partial v_{m}}{\partial t}, w_k)_2+(\nabla v_m, \nabla w_k)_2=(v_m|v_m|^{p-1}\log|v_m|, w_k)_2,\ k=1,2, \dots,
	\end{equation}
	\begin{equation}\label{approx_initial}
		v_m(0,t)=\sum\limits_{j=1}^m  w_j(x) g_{jm}(0) \to v_0(x) \ \textrm{in}\ H_0^1({U}).
	\end{equation}
	On substituting \eqref{approx1} in \eqref{approximate}--\eqref{approx_initial}, we find that $(g_{jm})_{j=1}^m$ solves an initial value problem whose solution is guaranteed in $[0,T_1]$ for some $T_1>0$.
	On multiplying \eqref{approximate} by $g_{km}^{\prime}$ and summing over $k$, we  get
	\begin{equation}\label{approx}
		\int\limits_0^t ||\frac{\partial v_{m}}{\partial t} (\cdot,\tau)||^2 d \tau+J(v_m(\cdot,t))=J(v_m(\cdot,0))<\delta, \quad 0 \leq t<T_1.
	\end{equation}
	We next show that there exists $  m_0 \in \mathbb{N}$ such that  $v_m(\cdot,t) \in W$, $m \geq m_0$, $0\leq t <T_1$. For, we observe that $v_m(\cdot,0) \in W,$ for sufficiently large $m$. On the contrary, assume  that $v_m(\cdot, \tilde{t}) \notin W$ for some $\tilde{t}$. Then one can easily get that either $v_m (\cdot, \tilde{t}) \in \mathcal{N}$ or $J(v_m(\cdot,\tilde{t}))=\delta$.  In the both cases we obtain $J(v_m(\cdot,\tilde{t}))\geq \delta$, which is a contradiction. Therefore $v_m(\cdot,t) \in W$, and we have
	\begin{align}
		J(v_m(\cdot,t))&=\frac{1}{(p+1)^2}||v_m(\cdot,t)||_{p+1}^{p+1}+ \frac{1}{p+1}I(v_m(\cdot,t))+\frac{p-1}{2(p+1)}||\nabla v_m(\cdot,t)||^2 \label{new}
		\\
		&\geq \frac{p-1}{2(p+1)}||\nabla v_m(\cdot,t)||^2. \label{jm}
	\end{align}
Moreover, from \eqref{jm}, we conclude that there is no finite time blow up for $v_m$, i.e., $v_m(\cdot,t)$ exists for $0 \leq t <\infty$.
Furthermore,	using \eqref{approx}--\eqref{jm}, one readily obtains 
	\begin{equation}\label{newapp}
		\int\limits_0^t ||\frac{\partial v_{m}}{\partial t}(\cdot,\tau)||^2 d \tau+\frac{p-1}{2(p+1)}||\nabla v_m(\cdot,t)||^2<\delta, \quad 0 \leq t<\infty,\ m \geq m_0.
\end{equation}

	% Thus from \eqref{newapp} we can get 
	%	\begin{equation}\label{nablau}
	%	||\nabla u_m||^2<\frac{2(p+1)}{p-1}d, \quad 0 \leq t<\infty,
	%\end{equation}
	%	\begin{equation}\label{intu}
	%	\int\limits_0^t ||u_{mt}||^2d t <\delta, \quad 0 \leq t<\infty,
	%\end{equation}
%	On the other hand, from the fact  $v_m(\cdot,t) \in W$, we get $I(v_m(\cdot,t))>0$, or
	
%	\begin{align}
%		\int\limits_{{U}} |v_m(x,t)|^{p+1}\log|v_m(x,t)| dx  &< ||\nabla v_m(\cdot,t)||^{2} \nonumber
%		\\
%		& <\frac{2(p+1)}{p-1}d, \quad 0 \leq t<\infty. \label{logu}
%	\end{align}
\noindent Let  ${U}_1= \{(x,t) \in {U} : |v_m(x,t)| \geq 1\}$. For the choice $r =1+\frac{1}{2p+1}>1$, and using \eqref{approx} --\eqref{new} and Lemma \ref{lemma_log},  we deduce % we have % and observe that
\begin{align}
	\int\limits_{{U}_1} (|v_m(x,t)|^{p}\log|v_m(x,t)|) ^r dx  & \leq  \int\limits_{{U}_1} (\frac{2}{e}|v_m(x,t)|^{p+\frac{1}{2}}) ^r dx \nonumber
	\\
	& = \big(\frac{2}{e}\big)^r \int\limits_{{U}_1} |v_m(x,t)| ^{p+1} dx \nonumber
	\\
	& \leq  ||v_m||_{p+1}^{p+1} \big(\frac{2}{e}\big)^r \nonumber
	\\
	& \leq  \delta (p+1)^2 \big(\frac{2}{e}\big)^r, \quad 0 \leq t<\infty.  \label{logu_int2}
\end{align}
Moreover, for ${U}_2 =\{(x,t) \in {U} : |v_m(x,t)| < 1\}$, we estimate
\begin{align}
	\int\limits_{{U}_2} (-|v_m(x,t)|^{p}\log|v_m(x,t)|)^r dx  &  \leq \frac{meas({U}_1)}{(pe)^r}, \quad 0 \leq t<\infty. \label{logu_int1}
\end{align}
Now using \eqref{logu_int1}--\eqref{logu_int2}, one can easily get that
	\begin{align}
		\int\limits_{{U}} \big||v_m(x,t)|^{p+1} \log|v_m(x,t)| \big|^rdx  &  \leq \frac{meas({U}_2)}{(pe)^r} +  \delta (p+1)^2\big(\frac{2}{e}\big)^r, \quad 0 \leq t<\infty. \label{logu_int3}
	\end{align}
	From \eqref{newapp} and \eqref{logu_int3} it follows that there exist a function $v$ and a subsequence of $(v_m),$ which is still denoted by $(v_m)$, such that
	\begin{equation}
		\left\{
		\begin{aligned}
			&v_m \overset{w^*}{\rightharpoonup} v\ \textrm{in} \ L^\infty  (0, \infty ; H^1_0({U})) \ \textrm{and}\ \textrm{a.e.}\ \textrm{in}\  {U} \times [0,\infty),
			\\
			&\frac{\partial v_{m}}{\partial t} \overset{w^*}{\rightharpoonup} v_t\ \textrm{in} \ L^2(0,\infty ; L^2({U})),
			\\
			&v_m|v_m|^{p-1}\log |v_m| \overset{w^*}{\rightharpoonup} v|v|^{p-1}\log |v| \ \textrm{in} \ L^\infty  (0, \infty ; L^r({U})),
		\end{aligned}
		\right.
	\end{equation}
	as $m \to \infty$. In \eqref{approximate}, for fixed  $k$, let $m \to \infty$ to deduce
	\begin{equation*}
		(\nabla v, \nabla w_k)_2+(v_{t}, w_k)_2=(v|v|^{p-1}\log|v|, w_k)_2.
	\end{equation*}
	From \eqref{approx_initial}, we have $v(x,0)=v_0(x)$ in $H_0^1({U})$. Hence $v$ is a global weak solution to \eqref{main}. Moreover, we get $v(x,t)\in W$ for $0\leq t<\infty$ from Theorem \ref{invariant}.\\
	\textbf{Step 2}: In this step, we prove decay estimate \eqref{decay}.\\
	From the fact that $I(v(\cdot, t))>0$, we have $0< J(v(\cdot, t))<\delta$. Therefore from Theorem \ref{invariant}, we get $I_\rho (v(\cdot, t)) >0$ for $0<\rho<1$. In other words, if $1-\beta=\rho$, then it follows that
	\begin{equation*}
		\int\limits_{{U}} |v(x,t)|^{p+1} \log |v(x,t)| dx<(1-\beta)||\nabla v(\cdot, t)||^2,
	\end{equation*}
	or
	\begin{equation}\label{loguu}
		\beta ||\nabla v(\cdot, t)||^2 < I(v(\cdot, t)).
	\end{equation}
	On the other hand, it is easy to observe that
	\begin{equation}\label{U_derivative}
		\frac{d}{dt} ||v(\cdot, t)||^2=-2I(v(\cdot, t)).
	\end{equation}
	Now using \eqref{loguu}--\eqref{U_derivative},  we obtain
	\begin{equation}\label{u_p}
		\begin{aligned}
			\frac{d}{dt} ||v(\cdot, t)||^2=-2I(v(\cdot, t))<-2 \beta ||\nabla v(\cdot, t)||^2 \leq -2\beta  \lambda_1 ||v(\cdot, t)||^2,
		\end{aligned}
	\end{equation}
	where $\lambda_1$ is the optimal  constant in the Poincar\'e inequality. Finally, Gronwall's lemma gives
	\begin{equation*}
		||v(\cdot, t)|| \leq  ||v_0|| e^{-\gamma t},\ t \geq 0,
	\end{equation*}
	where $\gamma= {\beta \lambda_1}$. This completes the proof.
\end{proof}
%%%%%%%%%%%%%%%%%%%%%%%%%%%%%%%%%%%%%
%%%%%%%%%%%%%%%%%%%%%%%%%%%%%%%%%%%
We next prove a   decay estimate of the solutions to \eqref{main} which is  stronger than the one given in  \eqref{decay}. In particular, we prove exponential decay of $H^1_0({U})$ norm of $v$.
\begin{proposition}\label{nabla_bound}
	Let the  power index $p$ in the source term satisfy
	%Let  the power index $p$ of the source term satisfy
	\begin{equation}\label{p}
		\begin{aligned}
			1<p< 
			\begin{cases}
				\infty, & \text{if } n \leq 2,\\
				\frac{n+2}{n-2}, & \text{if } n>2.
			\end{cases}
		\end{aligned}
	\end{equation}
	Assume that  $I(v_0)>0$ and $J(v_0)<\delta$,  then there exist constants $K>0$ and $\gamma >0$  such that the weak solution to \eqref{main} satisfies 
	\begin{equation}\label{decay_2}
		||\nabla v(\cdot, t)||\leq K e^{-\gamma t}, \quad 0 \leq t < \infty.
	\end{equation}
\end{proposition}
\begin{proof}
	In view of Theorem \ref{them1}, there exists a global solution to \eqref{main}. Let $v$ be  a global solution to \eqref{main}. 
	From conservation of energy \eqref{energy_inequality}, and using the fact that $v(\cdot,t) \in W, \ t\geq 0$, we deduce
	%	\begin{equation}
	\begin{align}
		J(v_0)\geq J(v(\cdot,t))&=\frac{1}{(p+1)^2}||v(\cdot, t)||_{p+1}^{p+1}+ \frac{1}{p+1}I(v(\cdot, t))+ \frac{p-1}{2(p+1)}||\nabla v(\cdot, t)||^2 \nonumber
		\\
		& \geq \frac{p-1}{2(p+1)}||\nabla v(\cdot, t)||^2. \label{j0}
	\end{align}
	%	\end{equation}
	From the hypothesis on $p$, and the Sobolev embedding theorem there exists $C_1>0$ such that 
	\begin{equation}\label{u_p+1}
		\begin{aligned}
			||v(\cdot, t)||_{p+1}\leq C_1||\nabla v(\cdot, t)||,\ t \geq 0.
		\end{aligned}
	\end{equation}
	Now using \eqref{loguu}, \eqref{u_p+1} and \eqref{j0}, we obtain
	\begin{equation}\label{u_p}
		\begin{aligned}
			||v(\cdot, t)||^{p+1}_{p+1} \leq C_1^{p+1} \left(\frac{2(p+1)}{p-1}J(v_0)\right)^{\frac{p-1}{2}}||\nabla v(\cdot, t)||^2<\theta I(v(\cdot, t)),
		\end{aligned}
	\end{equation}
	where $\theta=\frac{C_1^{p+1}}{\beta} \left(\frac{2(p+1)}{p-1}J(v_0)\right)^{\frac{p-1}{2}}$. 
	By \eqref{combo}, \eqref{loguu} and \eqref{u_p}, we get
	\begin{equation}\label{jtotal}
		\begin{aligned}
			J(v(\cdot, t)) &< \left(\frac{\theta}{ (p+1)^2} +\frac{1}{p+1}+  \frac{p-1}{2\beta (p+1)}\right) I(v(\cdot, t))=C I(v(\cdot, t)),
		\end{aligned}
	\end{equation}
	where $C=\left( \frac{p-1}{2\beta (p+1)}+\frac{1}{p+1}+\frac{\theta}{(p+1)^2}\right)$.
	Since  $\frac{d}{dt} ||v||^2=-2I(v(\cdot, t))$,  integration over $[t,T]$ yields
	\begin{equation}\label{iint}
		\begin{aligned}
			\int\limits_t^T I(v(\cdot, t))dt &=\frac{1}{2}\int\limits_{{U}} |v(\cdot, t)|^2 dx-\frac{1}{2}\int\limits_{{U}} |v(\cdot, T)|^2 dx
			\leq \frac{1}{2}||v(\cdot, t)||^2
			\leq \frac{1}{2 \lambda_1} ||\nabla v(\cdot, t)||^2,
		\end{aligned}
	\end{equation}
	where $\lambda_1$ is the optimal constant in  the Poincar\'e inequality.
	By \eqref{loguu} and \eqref{iint}, we have
	\begin{equation*}
		\begin{aligned}
			\int\limits_t^T I(v(\cdot, t))dt  < \frac{1}{2\lambda_1 \beta}I(v(\cdot, t)), \ T>0,
		\end{aligned}
	\end{equation*}
	which immediately gives 
	\begin{equation}\label{total}
		\begin{aligned}
			\int\limits_t^\infty I(v(\cdot, t))dt  < \frac{1}{2\lambda_1 \beta}I(v(\cdot, t)).
		\end{aligned}
	\end{equation}
	%which immediately gives 
	%$$\int_t^\infty I(v(\cdot, t)) dt < AI(v(\cdot, t)), \ t \geq 0.$$ 
	Now  \eqref{combo} and \eqref{jtotal} together imply
	\begin{equation*}
		\frac{1}{C}\int_t^\infty J(v(\cdot, t)) dt <\frac{p+1}{2 \lambda_1 \beta}J(v(\cdot, t)), \ 0\leq t \leq T.
	\end{equation*}
	Now choose $A=\frac{C(p+1)}{2 \lambda_1 \beta}$ and observe  from \eqref{energy_inequality} that $t \mapsto J(v(\cdot,t))$ is a non-increasing function. Due to  Lemma \ref{lemma_ineq}, one can arrive at
	$$J(v(\cdot, t)) \leq e^{1-\frac{t}{A}} J(v_0), \ t \geq 0.$$
	Using \eqref{combo} and \eqref{loguu}, one can conclude that there exist constants $K>0$ and $\gamma >0$ such that
	$$||\nabla v(\cdot,t)||\leq K e^{-\gamma t},\    t \geq 0.$$
	This completes the proof.
\end{proof}
%%%%%%%%%%%%%%%%%%%%%%%%%%%%%%%%%%%%%%%%%%%%%
%%%%%%%%%%%%%%%%%%%%%%%%%%%%%%%%%%%%%%%%%%%%
\begin{theorem}\label{them2}
	Let  $I(v_0)<0$ and $J(v_0)<\delta$. Then any weak solution to \eqref{main} blows up in finite time, i.e., there exists $T>0$ such that $$\lim\limits_{t \to T^{-}} ||v(\cdot,t)||=\infty.$$
\end{theorem}
\begin{proof}
	Let $v(x,t)$ be any weak solution to \eqref{main} with  $I(v_0)<0$ and $J(v_0)<\delta$.\\
	Define the function ${N} \colon [0,\infty) \to \mathbb{R}^+$  by ${N}(t)=\int\limits_0^t ||v(\cdot, \tau)||^2 d \tau $. Then an easy computation yields 
	\begin{equation}\label{m_derivative}
		\dot{{N}}(t)=||v(\cdot, t)||^2,\ \ddot{{N}}(t)=-2I(v(\cdot, t)).
	\end{equation}
	From \eqref{combo}, conservation of energy \eqref{energy_inequality}, and the Poincar\'{e} inequality, there exists $\lambda_1 >0$ such that 
	%	\begin{equation}
	\begin{align}
		\ddot{{N}}(t)&=(p-1)||\nabla v(\cdot,t)||^2-2(p+1)J(v(\cdot, t))+\frac{2}{p+1}||v(\cdot,t)||^{p+1}_{p+1} \nonumber
		\\
		&\geq (p-1)\lambda_1\dot{{N}}(t)+2(p+1)\int\limits_0^t ||v_t(\cdot, \tau)||^2 d \tau -2(p+1)J(v_0). \label{mdd}
	\end{align}
	%	\end{equation}
	Since 
	%	\begin{equation}
	\begin{align}
		\left(\int\limits_0^t (v_t (\cdot, \tau), v(\cdot, \tau))_2 d \tau \right)^2&=\left(\frac{1}{2}\int\limits_0^t \frac{d}{d t} ||v(\cdot,\tau)||^2 d \tau \right)^2 \nonumber
		\\
		&=\frac{1}{4}\left( \dot{{N}}^2(t)-2||v_0||^2 \dot{{N}}(t)+||v_0||^4\right), \label{mint}
	\end{align}
	%	\end{equation}
	we have
	\begin{equation*}
		\begin{aligned}
			&{N}\ddot{{N}}- \frac{p+1}{2}\dot{{N}}^2
			\\
			&\geq 2(p+1) \left\{ \int\limits_0^t ||v(\cdot, \tau)||^2 d \tau \int\limits_0^t ||v_t(\cdot, \tau)||^2 d \tau -\left(\int\limits_0^t (v_t(\cdot, \tau), v(\cdot, \tau))_2 d \tau\right)^2\right\}
			\\
			&\quad \quad -2(p+1)J(v_0){N}
			+(p-1)\lambda_1 {N} \dot{{N}}-(p+1)||v_0||^2 \dot{{N}}+\frac{p+1}{2}||v_0||^4.
		\end{aligned}
	\end{equation*}
	By H\"older's inequality, one can deduce that 
	\begin{equation}\label{bode}
		\begin{aligned}
			{N}\ddot{{N}}- \frac{p+1}{2}\dot{{N}}^2 &\geq (p-1)\lambda_1 {N} \dot{{N}}
			-(p+1)||v_0||^2 \dot{{N}}-2(p+1)J(v_0){N}.
		\end{aligned}
	\end{equation}
	\textbf{Claim:} For large enough $t>0$, it follows that
	%For sufficiently large $t$, it follows that
	\begin{equation}\label{ode}
		{N}\ddot{{N}}- \frac{p+1}{2}\dot{{N}}^2>0.
	\end{equation}
	To prove this claim, we consider two cases and argue separately.\\
	\textbf{Case-1:} Assume  $J(v_0)\leq 0.$ From  \eqref{mdd}, we get 
	$\ddot{{N}} \geq 0$. Since $\dot{{N}}(t)=||v(\cdot,t)||^2\geq 0$, then there exists  $t_0 \geq 0$ such that $\dot{{N}}(t_0)>0$ and 
	$${N}(t)\geq {{N}}(t_0)+\dot{{N}}(t_0)(t-t_0)> \dot{{N}}(t_0)(t-t_0), \quad t\geq t_0.$$
	Thus 
	%for sufficiently  large $t$,
	 we get $(p-1)\lambda_1 {N}>(p+1)||v_0||^2$,  whenever $t$ is large enough and \eqref{ode} follows immediately from \eqref{bode}.
	\\
	\textbf{Case-2:} Assume that $0<J(v_0)<\delta.$ 
	From Theorem \ref{invariant}, it is straightforward to obtain  $v(\cdot,t) \in V_\rho $ for $t>0$, $1< \rho<\rho_2$,  where $\rho_2$ is the same as the one introduced in  Theorem \ref{invariant}. In other words, we get  $I_{\rho}(v(\cdot,t))<0$,  for  $ t \geq 0$, $1< \rho<\rho_2$. Next we prove  that $||\nabla v(\cdot,t) ||^2 > {\lambda_1}||v_0||^2>0,\ t \geq 0$. For, since $I(v(\cdot, t))<0$, from \eqref{U_derivative} we deduce that $t \mapsto ||v(\cdot, t)||^2, \ t \geq 0$ is a strictly increasing function. On the other hand, the Poincar\'e inequality  gives $||\nabla v(\cdot, t)||^2 \geq \lambda_1 ||v(\cdot, t)||^2> \lambda_1||v_0||^2>0$.
	From  \eqref{m_derivative} and the definition of $I_{\rho}$, we find that 
	\begin{equation*}
		\ddot{{N}}(t)= 2(\rho-1)||\nabla v(\cdot,t)||^2-2I_{\rho}(v(\cdot,t))\geq2(\rho-1)\lambda_1||v_0||^2> 0,
	\end{equation*}
	\begin{equation*}
		\dot{{N}}(t)\geq 2(\rho-1)\lambda_1||v_0||^2t+\dot{{N}}(0)\geq  2(\rho-1)\lambda_1||v_0||^2t,
	\end{equation*}
	and 
	\begin{equation*}
		{{N}}(t)\geq (\rho-1)\lambda_1||v_0||^2 t^2+{{N}}(0)\geq  (\rho-1)\lambda_1||v_0||^2.
	\end{equation*}
	Thus we have
	\begin{equation}\label{Ode_inequa}
		\left\{
		\begin{aligned}
			(p-1)\lambda_1 {N}(t)>2(p+1)||v_0||^2,
			\\
			(p-1)\lambda_1\dot{{N}}(t)>4(p+1)J(v_0),
		\end{aligned}
		\right.
	\end{equation}
for sufficiently large $t$.
	On substituting \eqref{Ode_inequa} in  \eqref{bode}, we conclude that  \eqref{ode} holds for all sufficiently large $t$ which proves Claim.\\
	On the other hand, observe that 
	\begin{equation*}
		\frac{d^2}{dt^2} \left({N}^{-\frac{p-1}{2}}\right)	=-\frac{p-1}{2} {N}^{-\frac{p+3}{2}} \left({N}\dot{{N}}-\frac{p+1}{2}\dot{{N}}^2\right)<0,
	\end{equation*}
	for all sufficiently large $t$ due to \eqref{ode}. Hence for $t >\tilde{t}$, we can write 
	\begin{equation*}
		{N}^{-\frac{p-1}{2}}(t) < {N}^{-\frac{p-1}{2}}(\tilde{t})\left(1-\left(\frac{p-1}{2}\right)\frac{\dot{{N}}(\tilde{t})}{{N}(\tilde{t})}(t-\tilde{t})\right),
	\end{equation*} 
	which implies that there exists $T>0$ such that 
	$$\lim\limits_{t \to T^{-}}{N}^{-\frac{p-1}{2}}(t)=0,$$
	which completes the proof.
\end{proof}
%%%%%%%%%%%%%%%%%%%%%%%%%%%%%%%%%%%%%%%%%%%%%%%%%%
%%%%%%%%%%%%%%%%%%%%%%%%%%%%%%%%%%%%%%%%%%%%%%%%%% 
%%%%%%%%%%%%%%%%%%%%%%%%%%%%%%%%%%%
%%%%%%%%%%%%%%%%%%%%%%%%%%%%%%%%%%%%%%%%%%%%%%%%%%
%%%%%%%%%%%%%%%%%%%%%%%%%%%%%%%%%%%%%%%%%%%%%%%%%%
\section{The critical case $J(v_0)=\delta$}\label{sec:The critical case}
In this section, we discuss the global existence and finite time blow up of the solutions to \eqref{main} at the critical initial energy level $J(v_0)=\delta$.
%\subsection{Global existence at $J(v_0)=\delta$}
\begin{theorem}\label{them3}
	Assume that   $I(v_0)\geq 0$ and $J(v_0)=\delta$, then problem \eqref{main} admits a global weak solution $v \in L^\infty (0,\infty; H^1_0({U}))$. Moreover, we have $v(\cdot, t) \in \overline{W}$ for $0\leq t<\infty$. Furthermore, if $I(v_0)>0$ then there exist two positive constants $C$ and $\gamma $ such that
	\begin{equation}\label{decay_2nd}
		||v(\cdot,t)||\leq C e^{-\gamma t}, \quad 0 \leq t < \infty.
	\end{equation}
\end{theorem}
\begin{proof}
	Let $\mu_m =1-\frac{1}{m}$ and $v_{0m}=\mu_m v_0$, $m=2,3,\dots$. We consider the following problem 
	\begin{equation}\label{new_eq}
		\left\{
		\begin{aligned}
			&v_t-\rho v=v|v|^{p-1}\log|v|, &&x \in {U}, t>0,
			\\
			&v(x,t)=0, &&x \in \partial {U}, t>0,
			\\
			&v(x,0)=v_{0m}(x), && x \in {U}.
		\end{aligned}
		\right.
	\end{equation}
	Since $I(v_0) \geq 0$, in view of Lemma \ref{lemma1}, we have $\mu^*=\mu^*(v_0) \geq 1$. This immediately gives $I(v_{0m})>0$ and $J(v_{0m})=J(\mu_m v_0)< J(v_0)$. Moreover, we notice that 
	\begin{equation*}
		\begin{aligned}
			J(v_{0m})=\frac{p-1}{2(p+1)}||\nabla v_{0m}||^2+ \frac{1}{p+1}I(v_{0m})+\frac{1}{(p+1)^2}||v_{0m}||_{p+1}^{p+1}>0.
		\end{aligned}
	\end{equation*}
	Thus from Theorem \ref{them1}, it follows  that for each $m$, problem \eqref{new_eq} admits a global solution $v_m  \in L^\infty (0,\infty; H^1_0({U}))$ with $\frac{\partial v_{m}}{\partial t} \in  L^2 (0,\infty; L^2({U}))$, and $v_{m}(\cdot,t)\in W$ for $0\leq t <\infty$. In other words, we have
	\begin{equation*}
		(\frac{\partial v_{m}}{\partial t} , v)_2+(\nabla v_m, \nabla v)_2=(v_m|v_m|^{p-1}\log|v_m|, v)_2, \ \textrm{for any}  \ v\in H^1_0({U}),\ t\in(0,T),
	\end{equation*}
	and
	\begin{equation}\label{new_approx}
		J(v_m(\cdot,t))+\int\limits_0^t ||\frac{\partial v_{m}}{\partial t}(\cdot,\tau) ||^2d \tau = J(v_m(\cdot,0))<\delta, \quad 0 \leq t<\infty.
	\end{equation}
	On the other hand, from  Theorem \ref{invariant}, we deduce that $I(v_m(\cdot,t))>0$. By following the arguments presented in Theorem \ref{them1}, one can easily prove that   \eqref{newapp} and \eqref{logu_int3} hold for each $m$, and there exist a function $v$ and a subsequence of $(v_m)$ which is still denoted by $(v_m)$, such that
	\begin{equation}
		\left\{
		\begin{aligned}
			&v_m \overset{w^*}{\rightharpoonup} v\ \textrm{in} \ L^\infty  (0, \infty ; H^1_0({U})) \ \textrm{and}\ \textrm{a.e.}\ \textrm{in}\  {U} \times [0,\infty),
			\\
			&\frac{\partial v_{m}}{\partial t} \overset{w^*}{\rightharpoonup} v_t\ \textrm{in} \ L^2(0,\infty ; L^2({U})),
			\\
			&v_m|v_m|^{p-1}\log |v_m| \overset{w^*}{\rightharpoonup} v|v|^{p-1}\log |v| \ \textrm{in} \ L^\infty  (0, \infty ; L^r({U})),
		\end{aligned}
		\right.
	\end{equation}
	as $m \to \infty$. Now it is straightforward to get that $v$ is indeed a global solution and $v(\cdot, t) \in \overline{W}$ for $0\leq t<\infty$. \\
	\textbf{Decay estimate}\\
	Assume that $v$ is a global solution to  \eqref{main} with $I(v_0) > 0$, $J(v_0)=\delta$,  then we get $I(v(\cdot,t)) \geq 0$ for $0 \leq t < \infty$. We complete the proof  by considering following two cases.\\
	\textbf{Case 1.}  Assume that $I(v(\cdot, t))>0$ for $0 \leq t <\infty$. Then from the relation $(v_t, v)=-I(v(\cdot, t))<0$, it follows that $||v_t||>0$ and $\int_0^t ||v_t||^2 d\tau$ is strictly increasing  in $[0,\infty)$. Therefore from  \eqref{energy_inequality}, we obtain
	\begin{equation}
		J(v(\cdot,t_1))=-\int_0^{t_1} ||v_t(\cdot, \tau)||^2 d \tau+J(v_0)<\delta.
	\end{equation}
	Using the arguments that are employed in the proof of the decay estimate in Theorem  \ref{them1}, we can obtain the exponential decay \eqref{decay_2nd}.\\
	\textbf{Case 2.} Let if possible there exists  $t_1>0$ such that $I(v(\cdot, t))>0$ for $0 \leq t< t_1$ and $I(v(\cdot,t_1))=0$. Now two possibilities can arise, they are: $(i)$ $|| \nabla v(\cdot,t_1)||=0$, $(ii)$ $|| \nabla v(\cdot,t_1)||> 0$.\\ We now prove that $||\nabla v(\cdot,t_1)||> 0$ can not hold. For, it is enough to show that if $|| \nabla v(\cdot,t_1)||> 0$ then $I(v(\cdot, t_1))>0$.\\
	\textbf{Claim.} If $||\nabla v(\cdot,t_1)|| >0$ then $I(v(\cdot,t_1))>0$.\\
	As $(v_t, v)=-I(v(\cdot, t))$ it follows that $t \mapsto \int_0^t ||v_t||^2 dt$ is strictly increasing, for $0\leq t<t_1$ .  Owing to \eqref{energy_inequality}, we obtain
	\begin{equation}\label{new12}
		J(v(\cdot,t_1))=-\int_0^{t_1} ||v_t(\cdot, \tau)||^2 d \tau+J(v_0)<\delta.
	\end{equation}
	Since $||\nabla v(\cdot,t_1)|| >0$ and $I(v(\cdot,t_1))=0$, from the definition of $\delta$ we have $J(v(\cdot,t_1)) \geq \delta$, which is contradiction to  \eqref{new12}. This proves Claim.\\
	Therefore we conclude $||\nabla v(\cdot, t_1)|| =0$. Hence one can easily deduce that $v$ satisfies   \eqref{decay_2nd}.
\end{proof}
%%%%%%%%%%%%%%%%%%%%%%%%%%%%%%%%%%%%%%%%%%%%%%%%%%
%%%%%%%%%%%%%%%%%%%%%%%%%%%%%%%%%%%%%%%%%%%%%%%%%%
As we have done in the subcritical case, we prove a result pertaining to the asymptotic behavior of  the solutions to \eqref{main} which is stronger than \eqref{decay_2nd} under an additional assumption on $p$.
\begin{proposition}\label{nabla_bound_1}
	Let  the power index $p$ of the source term satisfy \eqref{p}.
	Assume that $I(v_0)>0$ and $J(v_0)=\delta$,   then there exist two positive constants $K$ and $\gamma $  such that the weak solution to  \eqref{main} satisfies 
	\begin{equation}\label{decay_3}
		||\nabla v(\cdot,t)||\leq K e^{-\gamma t}, \quad 0 \leq t < \infty.
	\end{equation}
\end{proposition}
\begin{proof}
	Using the arguments that are employed in the proof of the decay estimate in Theorem \ref{them3}, and the arguments used in Proposition \ref{nabla_bound}, we can  obtain the exponential decay \eqref{decay_3}.
\end{proof}
%%%%%%%%%%%%%%%%%%%%%%%%%%%%%%%%%%%%%%%%%%%%%%%%%%
%%%%%%%%%%%%%%%%%%%%%%%%%%%%%%%%%%%%%%%%%%%%%%%%%%%%%
\begin{theorem}\label{them4}
	Assume that  $I(v_0)<0$ and $J(v_0)=\delta$, then the weak solution to  \eqref{main} blows up in finite time i.e., there exists $T>0$ such that  $$\lim\limits_{t \to T^{-}} ||v(\cdot,t)||=\infty.$$
\end{theorem}
\begin{proof}
	Suppose $v$ is a weak solution to  \eqref{main} with $I(v_0)<0$ and $J(v_0)=\delta$. Moreover, assume that  $T$ is the existence time of $v$. We have to show that $T <\infty$. From the continuity of  $I(v(\cdot, t))$ and  $J(v(\cdot, t))$ as functions of  $t$, it follows that there exists a sufficiently small $t_1 \in (0,T)$ such that $I(v(\cdot,t))<0$ and $J(v(\cdot,t_1))>0$ for $0\leq t \leq t_1$. Therefore $t \mapsto \int\limits_0^t ||v_t(\cdot,\tau)||^2 d \tau $ is strictly increasing for $0\leq t\leq t_1$. From the conservation of energy \eqref{energy_inequality}, we can choose $t_1$ such that  
	\begin{equation}
		0<J(v(\cdot,t_1))=-\int\limits_0^{t_1} ||v_t(\cdot, \tau)||^2 d \tau+J(v_0)<J(v_0)=\delta.
	\end{equation}
In view of Theorem \ref{them2}, one can easily get that the maximal existence time $T$ of $v$ is finite, $i.e.,$ 
$$\lim\limits_{t \to T^{-}} ||v(\cdot,t)||=\infty.$$
This completes the proof.
\end{proof}
%%%%%%%%%%%%%%%%%%%%%%%%%%%%%%%%%%%%%%%%%%%%%%%%%%
%%%%%%%%%%%%%%%%%%%%%%%%%%%%%%%%%%%%%%%%%%%%%%%%%%
\section{Finite time blowup}\label{sec:finite_time_blow}
In the previous section, we have proved that the weak solution exhibits finite time blow up when $I(v_0)<0$ in the subcritical case and the critical case.  In this section, we would like to prove a similar result when $J(v_0)>\delta$. However, when the $L^2$ norm of the initial data is  sufficiently larger than the initial potential energy then we observe finite time blow up irrespective of the magnitude of $J(v_0)$. Details are given in the following theorem.
\begin{theorem}
	Let the initial data $v_0$ satisfy\\
	$(i)$ $J(v_0)>\delta$,\\
	$(ii)$ $||v_0||^2>\frac{2(p+1)}{\lambda_1(p-1)} J(v_0)$,\\
	$(iii)$ $I(v_0)<0$,\\
	where $\lambda_1$ is the optimal constant in the Poincar\'e inequality. Then any weak solution to  \eqref{main} blows up in finite time.
\end{theorem}
\begin{proof}
	As in the proof of Theorem \ref{them2}, we work with the quantity ${N}(t)=\int\limits_0^t ||v(\cdot,\tau)||^2 d \tau $. The following two steps are used to prove the theorem.\\
	\textbf{Step 1.} In this step, we show  $|| v(\cdot,t)||^2 > \frac{2(p+1)}{\lambda_1(p-1)} J(v_0)$ and $I(v(\cdot, t))<0$,  $t \in (0, T)$.\\
	Suppose there exists $t_0 \in (0,T)$ such that $I(v(\cdot,t))<0$ for $0\leq t <t_0$ and $I(v(\cdot,t_0)) =0$. From the definition it is clear that ${N}$ is increasing, and $\ddot{{N}}(t)= - 2I(v(\cdot, t))>0$ for $t\in [0,t_0)$. Therefore,  $\dot{{N}}$ is increasing in $[0,t_0]$. Thus we obtain
	\begin{equation}\label{mt0}
		\dot{{N}}(t_0)>\dot{{N}}(0)=||v_0||^2> \frac{2(p+1)}{\lambda_1(p-1)} J(v_0).
	\end{equation}
	Since $I(v(\cdot,t_0))=0$, the conservation of energy gives 
	\begin{equation*}
		\begin{aligned}
			J(v_0)&\geq J(v(\cdot,t_0))
			\\
			&=\frac{p-1}{2(p+1)}||\nabla v(\cdot,t_0)||^2+ \frac{1}{(p+1)^2}||v(\cdot,t_0)||_{p+1}^{p+1}
			\\
			&\geq \frac{p-1}{2(p+1)}||\nabla v(\cdot,t_0)||^2
			\\
			&\geq  \frac{ (p-1)\lambda_1}{2(p+1)}|| v(\cdot,t_0)||^2.
		\end{aligned}
	\end{equation*}
	Hence we get $\dot{{N}}(t_0)=||v (\cdot,t_0)||^2 \leq \frac{2(p+1)}{\lambda_1 (p-1)} J(v_0)$, which is a contradiction to \eqref{mt0}. Therefore we have
	\begin{equation*}
		I(v(\cdot, t))<0, \  t\in (0,T),
	\end{equation*}
	and \eqref{mt0} holds for every $t \in (0,T)$.
	%\begin{equation}\label{m}
	%	\dot{{N}}(t)> ||v_0||^2> \frac{4(p+1)}{\lambda_1 (p-1)} J(v_0),\  t\in (0,T),
	%\end{equation}
	Consequently, it implies that ${N}$ is strictly increasing. Therefore if $t$ is large enough, we get
	\begin{equation}\label{mm}
		{N}(t)>\frac{2(p+1)}{ (p-1)\lambda_1} ||v_0||^2.
	\end{equation}
	\textbf{Step 2.} In this step, for sufficiently large $t$, we prove  $\ddot{{N}}{N}-\frac{p+1}{2}\dot{{N}}^2>0$.\\
	For, from \eqref{bode}, \eqref{mt0} and \eqref{mm}, we deduce that
	\begin{equation}
		\begin{aligned}
			\ddot{{N}}{N}-\frac{p+1}{2}\dot{{N}}^2
			&\geq \lambda_1(p-1) {N} \dot{{N}}
			-2(p+1)J(v_0){N}-(p+1)||v_0||^2 \dot{{N}} >0,
		\end{aligned}
	\end{equation}
for sufficiently large $t$, proving Step 2.
	By considering ${N}^{-\frac{p-1}{2}}$ and using the convexity arguments that are  presented in the proof of Theorem \ref{them2}, it is straightforward to show the finite time blow up of the weak solution $v$.
\end{proof}
\section*{Conclusions}
We have proved the existence of  global  solutions to the initial value problem of a semi-linear heat equation  \eqref{main} without having any restriction on $p$ and the dimension at two different energy levels (subcritical and critical) provided $I(v_0) > 0$. Moreover, an exponential decay estimate on $L^2$ - norm of the global solutions has been obtained  for all $p>1$ and $n \in \mathbb{N}$.  We have also estimated the $H^1_0$ - norm of the solution under the condition $1<p<\frac{n+2}{n-2}, \ \text{if } n>2$.
On the other hand, we have established that if $I(v_0)<0$, then any solution exhibits the finite time blow property at subcritical and critical initial energy levels. Besides we have proved that solution blows up in finite time, provided $|| v_0||$ is sufficiently larger than $J(v_0) > \delta$, and $I(v_0)<0$. It is an interesting problem to prove that any solution exhibits the finite time blow property without the condition $|| v_0|| >> J(v_0)>0$ in the supercritical case ($J(v_0) >\delta$). Moreover, investigation of the global existence in the super critical case is also an interesting problem.
\section*{Acknowledgements}
\noindent The first author would like to thank CSIR (award number: 09/414(1154)/2017-EMR-I) for providing the financial support for his research.
The second author is supported by Department of Science and Technology, India, under MATRICS (MTR/2019/000848).

%%%%%%%%%%%%%%%%%%%%%%%%%%%%%%%%%%%%%%%%%%%%%%%%%%%
\bibliographystyle{abbrv}
\bibliography{references}
\end{document}